\documentclass[final,leqno]{siamltex}
\usepackage{amsmath}
\usepackage{amssymb}
\usepackage{float}
\usepackage{comment}
\usepackage{graphicx}
\usepackage{textcomp}
\usepackage{srcltx} 
\usepackage{hyperref} 
\hypersetup{draft=false, colorlinks=true, linkcolor=black, urlcolor=blue, citecolor=black}
\newtheorem{remark}[theorem]{Remark}
\newtheorem{algorithm}[theorem]{Algorithm}

\newcommand{\RR}{\mathbb{R}}

\newcommand{\NN}{\mathbb{N}}
\newcommand{\ctc}{c_{\mathrm{tc}}}

\author{Anne Wald\thanks{Department of Mathematics, Saarland University, PO Box 15 11 50, 66041 Saarbr\"ucken, Germany ({\tt wald@math.uni-sb.de}).}
        \and Thomas Schuster\thanks{Department of Mathematics, Saarland University, PO Box 15 11 50, 66041 Saarbr\"ucken, Germany ({\tt thomas.schuster@num.uni-sb.de}), correspondent author.}}      
     
\title{Sequential subspace optimization for nonlinear inverse problems}
        
\begin{document}

\maketitle

\begin{abstract}
 In this work we discuss a method to adapt sequential subspace optimization (SESOP), which has so far been developed for linear inverse problems in Hilbert and Banach spaces, to the case of nonlinear inverse problems. We start by revising the well-known technique for Hilbert spaces. In a next step, we introduce a method using multiple search directions that are especially designed to fit the nonlinearity of the forward operator. To this end, we iteratively project the initial value onto stripes whose shape is determined by the search direction, the nonlinearity of the operator and the noise level. We additionally propose a fast algorithm that uses two search directions. Finally we will show convergence and regularization properties for the presented method.
\end{abstract}


\section{Introduction}
In this article, we are dealing with an iterative solution of nonlinear inverse problems in Hilbert spaces. We consider the operator equation
\begin{equation} \label{operator_equation}
 F(x) = y
\end{equation}
for a nonlinear operator $F: \mathcal{D}(F)\subset X \rightarrow Y$ between Hilbert spaces $X$ and $Y$. If only noisy data $y^{\delta}$ are available, we assume that the noise level $\delta$ fulfills
\begin{displaymath} 
 \left\lVert y^{\delta} - y \right\rVert \leq \delta.
\end{displaymath}
In general, equation \eqref{operator_equation} is ill-posed, i.~e.~the solution of \eqref{operator_equation} does not depend continuously on the data. The direct inversion of such an ill-posed problem will, in the case of noisy data, most probably lead to useless results. For this reason we need to apply regularizing techniques to approximate the inverse $F^{-1}$ and find a suitable stable solution. \\
There is a great range of methods that have been developed for linear operator equations (for references, see e.~g.~\cite{louis89}, \cite{ehn00},\cite{skhk12}), some of which have been successfully adapted to nonlinear inverse problems. An overview of such methods is given in \cite{kns_itreg}, \cite{riederkp}. Well-known regularization methods are for example the Tikhonov regularization (\cite{ekn89}, \cite{scherzer93}, \cite{kr93}, \cite{neubauer92}) or Newton type methods such as the iteratively regularized Gauss-Newton method (see \cite{kss09}) or the inexact Newton method (\cite{kss09}, \cite{rieder99}). The Landweber regularization has also been successfully modified for nonlinear problems and has been thoroughly studied (\cite{hns_cl}, \cite{kss09}, \cite{kaltenbacher97}). Finally, the conjugate gradient method is another example to mention \cite{hanke97}. In fact, most regularization techniques for nonlinear problems are iterative methods. \\
Many inverse problems arising from applications in the natural sciences, engineering and other fields are nonlinear. One field of interest is inverse scattering with its many different applications for example in medical and nondestructive testing (\cite{gbpl}, \cite{ck1}, \cite{ll08}, \cite{lakhal10}). \\
The method we present is inspired by an algorithm that was first introduced in \cite{gnmz05}, \cite{emz07} for linear problems between finite dimensional vector spaces. The key idea is to use more than one search direction in each step of the iteration instead of just one as it is the case for the Landweber or the conjugate gradient method, widening the subspace in which an approximation to the solution is sought. This leads to a more challenging optimization in each step of the iteration, but it may reduce the total number of steps to obtain a satisfying approximation \cite{ss09}. This is an advantage if each step requires the costly solution of the forward problem. The Landweber iteration requires the solution of one forward problem and of one adjoint linearized problem, while there is only one parameter that has to be determined by a subspace optimization. Using more than one search direction, we want to exploit fast subspace optimization schemes and thus reduce the total number of iterations. This approach can also be interpreted as a split feasibility problem (see \cite{cekb05}) as it has been discussed in \cite{skhk12}.\\
The algorithm we want to introduce in this paper is mainly based on the algorithm as discussed in \cite{ss09}, \cite{sls08} for linear operators between Banach spaces. Due to the local character of nonlinear problems, some of the statements cannot be transfered from the linear to the nonlinear case. By making some assumptions on the nonlinear forward operator, we can still find a way to modify the algorithm to fit our requirements. A main result is that for a certain choice of search directions, we can prove convergence and regularization results. \\
In Section \ref{mathematical_setup} we first want to summarize the method for linear problems in Hilbert spaces and give an overview of the relevant techniques. The method for nonlinear inverse problems is introduced in Section \ref{section_nonlinear}, were we begin with the basic definitions before moving on to the presentation of our reconstruction methods. Section \ref{section_convergence_regularization} contains the convergence analysis and a regularization result for the proposed method.


\section{Mathematical Setup} \label{mathematical_setup}
First of all, we want to give an overview of the Sequential Subspace Optimization (SESOP) method and the corresponding REgularized SESOP (RESESOP) method applied to linear equations as discussed in \cite{ss09}, including the RESESOP algorithm for two search directions. We restrict ourselves to the SESOP method for Hilbert spaces. Transferring the results for Banach spaces to Hilbert spaces, the duality mappings turn into identity mappings and instead of general Bregman distances and Bregman projections, we use the metric distance and metric projection. We identify the duals of the spaces $X$ and $Y$ with the spaces themselves and we will drop the subscripts for the norms $\|\cdot\|_X$, $\|\cdot\|_Y$ for a better readability whenever confusion is not possible.


\subsection{Sequential Subspace Optimization in Hilbert Spaces for Linear Problems}
Let $X$,$Y$ be Hilbert spaces and $A: X\rightarrow Y$ a linear operator. We consider the operator equation
\begin{equation} \label{linearproblem}
 Ax = y
\end{equation}
with the solution set
\begin{equation}
 M_{Ax=y} := \left\lbrace x \in X  :  Ax=y \right\rbrace
\end{equation}
and assume that noisy data $y^{\delta}$ are given. \\
We want to utilize an iteration of the form 
\begin{equation}
 x_{n+1} := x_n - \sum_{i \in I_n} t_{n,i} A^*w_{n,i}
\end{equation}
to calculate a solution $x \in X$, where $A^*$ is the adjoint operator of $A$ and $I_n$ is a finite index set. The parameters $t_n:=(t_{n,i})_{i\in I_n} \in \mathbb{R}^{\lvert I_n \rvert}$ minimize the function
\begin{equation}
 h_n(t) := \frac{1}{2} \left\lVert x_n - \sum_{i \in I_n} t_i A^*w_{n,i} \right\rVert^2 + \sum_{i\in I_n} t_i \left\langle w_{n,i}, y \right\rangle.
\end{equation}
In \cite{sls08} it was shown that the minimization of $h_n(t)$ is equivalent to computing the metric projection 
\begin{displaymath}
 x_{n+1} = P_{\bigcap_{i \in I_n} H_{n,i}}(x_n) 
\end{displaymath}
onto the intersection of hyperplanes
\begin{displaymath}
 H_{n,i} := \left\lbrace x \in X \ : \ \left\langle A^* w_{n,i}, x \right\rangle - \left\langle w_{n,i},y \right\rangle = 0 \right\rbrace.
\end{displaymath}
Note, that $M_{Ax=y}\subset H_{n,i}$ for all $i \in I_n$. This motivates the regularizing sequential subspace optimization methods where we replace the hyperplanes by stripes whose width is of the order of the noise level $\delta$.\\[1ex]

\begin{definition}
 The metric projection of $x \in X$ onto a nonempty closed convex set $C \subset X$ is the unique element $P_C(x) \in C$ such that
 \begin{displaymath}
  \left\lVert x - P_C(x) \right\rVert^2 = \min_{z \in C} \left\lVert x - z \right\rVert^2.
 \end{displaymath}
 For later convenience, we use the square of the distance.\\[1ex]
\end{definition}

The metric projection onto a convex set $C$ fulfills a descent property which reads
\begin{displaymath}
 \left\lVert z - P_C(x) \right\rVert^2 \leq \left\lVert z - x\right\rVert^2 - \left\lVert P_C(x) - x \right\rVert^2
\end{displaymath}
for all $z \in C$.\\[1ex]

\begin{definition}
 For $u \in X$, $\alpha, \xi \in \RR$ with $\xi \geq 0$ we define the hyperplane
 \begin{displaymath}
  H(u,\alpha) := \left\lbrace x \in X \ : \ \left\langle u,x \right\rangle = \alpha \right\rbrace,
 \end{displaymath}
 the halfspace
 \begin{displaymath}
  H_{\leq}(u,\alpha) := \left\lbrace x \in X \ : \ \left\langle u,x \right\rangle \leq \alpha \right\rbrace
 \end{displaymath}
 and analogously $H_{\geq}(u,\alpha)$, $H_{<}(u,\alpha)$ and $H_{>}(u,\alpha)$. Finally, we define the \emph{stripe}
 \begin{displaymath}
  H(u,\alpha,\xi) := \left\lbrace x \in X \ : \ \left\lvert \left\langle u,x \right\rangle - \alpha \right\rvert \leq \xi \right\rbrace.
 \end{displaymath}
\end{definition}

Obviously we can write $H(u,\alpha,\xi) = H_{\leq}(u,\alpha+\xi) \cap H_{\geq}(u,\alpha - \xi)$ and $H(u,\alpha,0) = H(u,\alpha)$, i.~e.~the stripe corresponding to $u$ and $\alpha$ with width $2\xi$ contains the hyperplane $H(u,\alpha)$. \\
The metric projection $P_{H(u,\alpha)}(x)$ of $x \in X$ onto a hyperplane $H(u,\alpha)$ corresponds in the Hilbert space setting to an orthogonal projection and we have
\begin{equation} \label{orthogonal_projection}
 P_{H(u,\alpha)}(x) = x - \frac{\langle u,x \rangle - \alpha}{\lVert u \rVert^2}u.
\end{equation}
\vspace*{2ex} \\
The proof of the following statement can be found in \cite{skhk12} for Banach spaces and can be easily transferred to the situation in Hilbert spaces. When projecting onto the intersection of multiple hyperplanes or halfspaces, these statements turn out to be helpful tools.\\[1ex]

\begin{proposition}\label{proposition_properties_projection}
 \begin{itemize}
  \item[(i)] Let $H(u_i,\alpha_i)$, $i=1,...,N$, be hyperplanes with nonempty intersection $H:=\bigcap_{i=1,...,N} H(u_i,\alpha_i)$. The projection of $x$ onto $H$ is given by
    \begin{equation}
      P_H(x) = x - \sum_{i=1}^N \tilde{t}_i u_i,
    \end{equation}
   where $\tilde{t}=(\tilde{t}_1,...,\tilde{t}_N)$ minimizes the convex function 
    \begin{equation}
      h(t) = \frac{1}{2} \left\lVert x - \sum_{i=1}^N t_i u_i \right\rVert^2 + \sum_{i=1}^N t_i \alpha_i
    \end{equation}
   with partial derivatives
   \begin{displaymath}
    \partial_j h(t) = -\left\langle u_j, x-\sum_{i=1}^N t_i u_i \right\rangle + \alpha_j.
   \end{displaymath}
   If the vectors $u_i$, $i=1,...,N$, are linearly independent, $h$ is strictly convex and $\tilde{t}$ is unique.\\
  \item[(ii)] Let $H_i:=H_{\leq}(u_i,\alpha_i)$, $i=1,2$, be two halfspaces with linearly independent vectors $u_1$ and $u_2$. Then $\tilde{x}$ is the projection of $x$ onto $H_1\cap H_2$ iff it safisfies the Karush-Kuhn-Tucker conditions for $\min_{z \in H_1\cap H_2} \lVert z-x \rVert^2$, which read
   \begin{equation}
    \begin{split}
     &\tilde{x} = x-t_1 u_1 -t_2 u_2, \ t_1,t_2 \geq 0,\\
     &\langle u_i,\tilde{x} \rangle \leq \alpha_i, \ i=1,2, \\
     &t_i\left( \alpha_i - \langle u_i,\tilde{x} \rangle \right) \leq 0, \ i=1,2.
    \end{split}
   \end{equation}
  \item[(iii)] For $x\in H_{>}(u,\alpha)$ the projection of $x$ onto $H_{\leq}(u,\alpha)$ is given by
   \begin{equation}
    P_{H_{\leq}(u,\alpha)}(x) = P_{H(u,\alpha)}(x) = x - t_+ u,
   \end{equation}
   where
   \begin{equation}
    t_+ = \frac{\langle u,x\rangle - \alpha}{\lVert u \rVert ^2} \ > 0.
   \end{equation}
  \item[(iv)] The projection of $x \in X$ onto a stripe $H(u,\alpha,\xi)$ is given by
   \begin{equation}
    P_{H(u,\alpha,\xi)}(x)=\begin{cases}
                               P_{H_{\leq}(u,\alpha+\xi)(x)}, & x\in H_>(u,\alpha+\xi), \\
                               x,                                 & x\in H(u,\alpha,\xi), \\
                               P_{H_{\geq}(u,\alpha-\xi)(x)}, & x\in H_<(u,\alpha-\xi).
                              \end{cases}
   \end{equation}
 \end{itemize}
\end{proposition}

The following algorithm provides a method to compute the metric projection
\begin{displaymath}
 P_{M_{Ax=y}}(x_0)
\end{displaymath}
of $x_0 \in X$ onto the solution set $M_{Ax=y}$ in the case of exact data $y \in \mathcal{R}(A)$.\\[1ex]

\begin{algorithm} \label{algo1}
 Choose an initial value $x_0 \in X$. At iteration $n \in \NN$, choose a finite index set $I_n$ and search directions $A^*w_{n,i}$ with $w_{n,i} \in Y$ and $i\in I_n$. Compute the new iterate as
 \begin{equation}
  x_{n+1} := P_{H_n}(x_n),
 \end{equation}
 where
 \begin{equation}
  H_n := \bigcap_{i \in I_n} H(A^* w_{n,i}, \langle w_{n,i},y \rangle).
 \end{equation}
\end{algorithm}

We have $M_{Ax=y} \subset H_n$  and for all $z\in M_{Ax=y}$ holds
\begin{equation}
 \langle w_{n,i}, Ax_{n+1} -y \rangle=\langle A^*w_{n,i}, x_{n+1}-z \rangle=0
\end{equation}
for all $i \in I_n$. We define the search space 
\begin{displaymath}
 U_n := \mathrm{span}\left\lbrace A^*w_{n,i} \ : \ i\in I_n \right\rbrace
\end{displaymath}
as the span of the search directions used in iteration $n$ and note that $(x_{n+1} - z) \bot U_n$.\\
As stated in Proposition \ref{proposition_properties_projection}, $x_{n+1}$ can be computed by minimizing the convex function $h$. The search directions $A^*w_{n,i}$ spanning the subspace in which a minimizing solution is sought are fixed, so the minimization does not require any costly applications of $A$ or its adjoint $A^*$. The additional cost due to higher dimensional search spaces is comparatively minor.\\
For weak convergence, the current gradient $A^*(Ax_n-y)$ of the functional $\frac{1}{2}\lVert Ax-y \rVert^2$ needs to be included in the search space to guarantee an estimate of the form
\begin{equation}
 \lVert z- x_{n+1} \rVert^2 \leq \lVert z-x_n \rVert^2 - \frac{\lVert R_n \rVert^2}{\lVert A \rVert^2}
\end{equation}
for all $z\in M_{Ax=y}$, where $R_n:=Ax_n-y$ is the current residual. \\


\subsection{Regularizing Sequential Subspace Optimization}
If only noisy data $y^{\delta}\in Y$ are given, we modify Algorithm \ref{algo1} to turn it into a regularizing method. We define the two canonical sets of search directions
\begin{equation}
 G_n^{\delta} := \left\lbrace A^*(Ax_k^{\delta} - y^{\delta}) \ : \ 0 \leq k \leq n \right\rbrace
\end{equation}
and
\begin{equation}
 D_n^{\delta} := \left\lbrace x_k^{\delta} - x_l^{\delta} \ : 0 \leq l < k \leq n \right\rbrace.
\end{equation}

\begin{algorithm} \label{algo2}
 Choose an initial value $x_0^{\delta}:=x_0$ and fix some constant $\tau>1$. At iteration $n \in \NN$, choose a finite index set $I_n^{\delta}$ and search directions $A^*w_{n,i}^{\delta} \in G_n^{\delta}\cap D_n^{\delta}$ as defined above. If the residual $R_n^{\delta}$ satisfies the discrepancy principle
 \begin{equation}
  \lVert R_n^{\delta} \rVert \leq \tau\delta,
 \end{equation}
 stop iterating. Otherwise, compute the new iterate as
 \begin{equation}
  x_{n+1}^{\delta} := x_n^{\delta} - \sum_{i \in I_n^{\delta}} t_{n,i}^{\delta} A^*w_{n,i}^{\delta}, 
 \end{equation}
 choosing $t_n^{\delta}=(t_{n,i}^{\delta})_{i \in I_n^{\delta}}$ such that
 \begin{equation} \label{iterate_proj2}
  x_{n+1}^{\delta} \in H_n^{\delta} := \bigcap_{i \in I_n^{\delta}} H\big(A^*w_{n,i}^{\delta}, \langle w_{n,i}^{\delta},y^{\delta} \rangle, \delta \lVert w_{n,i}^{\delta} \rVert\big)
 \end{equation}
 and such that an inequality of the form
 \begin{equation} \label{descent2}
  \lVert z-x_{n+1}^{\delta} \rVert^2 \leq \lVert z-x_n^{\delta} \rVert^2 - C\lVert R_n^{\delta} \rVert^2
 \end{equation}
 holds for all $z \in M_{Ax=y}$ for some constant $C>0$.\\[1ex]
\end{algorithm}

The choice $A^*w_{n,i}^{\delta} \in G_n^{\delta}\cap D_n^{\delta}$ exploits the linearity of the operator $A$, which yields a recursion for the computation of the search directions. This part can not be adapted to the nonlinear case. \\
However, we note that $M_{Ax=y}\subset H_n^{\delta}$ because for $z \in M_{Ax=y}$ we have
\begin{displaymath}
 \big\lvert \langle A^*w_{n,i}^{\delta}, z \rangle - \langle w_{n,i}^{\delta}, y^{\delta} \rangle \big\rvert = \big\lvert \langle w_{n,i}^{\delta}, y-y^{\delta} \rangle \big\rvert \leq \delta \lVert w_{n,i}^{\delta} \rVert.
\end{displaymath}
Due to \eqref{descent2}, the sequence $\lVert z - x_n^{\delta} \rVert$ decreases for a fixed noise level $\delta$, i.~e.~the discrepancy principle yields a finite stopping index $n_*=n_*(\delta):=\min\lbrace n \in \NN: \lVert R_n^{\delta} \rVert \leq \tau\delta \rbrace$.
\vspace*{2ex} \\
It is possible to prove convergence results and other interesting statements for certain choices of search directions for RESESOP in the linear case. We refer the reader to \cite{skhk12} and \cite{ss09}. Before looking at a suitable method for nonlinear inverse problems, we want to take a look at the important case of two search directions.\\


\subsection{RESESOP with two search directions}
Before looking at a suitable method for nonlinear inverse problems, we want to summarize this fast way to compute $x_{n+1}^{\delta}$ according to Algorithm \ref{algo2}, using only two search directions, such that \eqref{iterate_proj2} and \eqref{descent2} are valid. As illustrated in \cite{ss09}, the projection of $x \in X$ onto the intersection of two halfspaces can be computed by at most two projections, if $x$ is already contained in one of them. The following proposition provides the basic results for Algorithm \ref{algo3}.\\[1ex]

\begin{proposition} \label{prop_proj_onto_two}
 Let $H_j:=H_{\leq}(u_j,\alpha_j)$, $j=1,2$, be two halfspaces with $H_1\cap H_2 \neq \emptyset$. Let $x\in H_>(u_1,\alpha_1) \cap H_2$. The metric projection of $x$ onto $H_1 \cap H_2$ can be computed by at most two metric projections onto (intersections of) the bounding hyperplanes by the following steps:
 \begin{itemize}
  \item[(i)] Compute
   \begin{displaymath}
    x_1 := P_{H(u_1,\alpha_1)}(x) = x - \frac{\langle u_1,x \rangle - \alpha_1}{\lVert u_1 \rVert^2} u_1.
   \end{displaymath}
   Then, for all $z \in H_1$, we have
   \begin{equation} \label{descent_step1}
    \lVert z - x_1 \rVert^2 = \lVert z-x \rVert^2 - \left(\frac{\langle u_1,x \rangle - \alpha_1}{\lVert u_1 \rVert}\right)^2.
   \end{equation}
   If $x_1 \in H_2$, we are done. Otherwise, go to step (ii):
  \item[(ii)] \label{step2_prop_proj_onto_two} Compute
   \begin{displaymath}
    x_2 := P_{H_1 \cap H_2}(x_1).
   \end{displaymath}
   Then $x_2 = P_{H_1 \cap H_2}(x)$ and for all $z \in H_1 \cap H_2$ we have
   \begin{equation} \label{descent_step2}
    \lVert z - x_2 \rVert^2 \leq \lVert z - x \rVert^2 - \left( \left( \frac{\langle u_1,x \rangle - \alpha_1}{\lVert u_1 \rVert} \right)^2 + \left( \frac{\langle u_2,x_1 \rangle - \alpha_2}{\gamma \lVert u_2 \rVert} \right)^2 \right),
   \end{equation}
   where
   \begin{equation}
    \gamma := \left( 1-\left( \frac{\big\lvert \langle u_1,u_2 \rangle \big\rvert}{\lVert u_1 \rVert \cdot \lVert u_2 \rVert} \right)^2 \right)^{\frac{1}{2}} \ \in (0,1].
   \end{equation}
 \end{itemize}
\end{proposition}

\begin{proof}
 The proof of Proposition \ref{prop_proj_onto_two} can be found in \cite{ss09}. In our Hilbert space setting, equation \eqref{descent_step1} requires no further proof as it follows directly from equation \eqref{orthogonal_projection}. To obtain the descent property \eqref{descent_step2} in step (ii), the statements from Proposition \ref{proposition_properties_projection} are used to find suitable estimates.
\hfill
\end{proof}
\vspace{2mm}

The following choice of search directions along with Algorithm \ref{algo3} assures that the projections in each step can be calculated uniquely according to Proposition \ref{prop_proj_onto_two} and yields a fast regularizing method to compute a solution of \eqref{linearproblem} using noisy data.\\[1ex]

\begin{algorithm} \label{algo3}
 Choose an initial value $x_0^{\delta} \in X$. In the first step ($n=0$) take $u_0^{\delta}$ and then, for all $n\geq 1$, the search directions $\lbrace u_n^{\delta}, u_{n-1}^{\delta} \rbrace$ in Algorithm \ref{algo2}, where 
 \begin{displaymath}
  u_n^{\delta} := A^*w_n^{\delta}, w_n^{\delta}:= Ax_n^{\delta} - y^{\delta}.
 \end{displaymath}
 Define $H_{-1}^{\delta} := X$ and, for $n \in \NN$, the stripes
 \begin{displaymath}
  H_n^{\delta} := H\big(u_n^{\delta},\alpha_n^{\delta},\delta \lVert R_n^{\delta} \rVert \big)
 \end{displaymath}
 with
 \begin{displaymath}
  \alpha_n^{\delta} := \langle u_n^{\delta}, x_n^{\delta} \rangle - \lVert R_n^{\delta} \rVert^2.
 \end{displaymath}
 As long as $\lVert R_n^{\delta} \rVert > \tau\delta$ we have
 \begin{equation}
  x_n^{\delta} \in H_{>}\big( u_n^{\delta}, \alpha_n^{\delta} + \delta \lVert R_n^{\delta} \rVert \big) \cap H_{n-1}^{\delta}.
 \end{equation}
 Compute $x_{n+1}^{\delta}$ by the following steps:\\
 \begin{itemize}
  \item[(i)] Compute 
   \begin{displaymath}
    \tilde{x}_{n+1}^{\delta} := P_{H(u_n^{\delta},\alpha_n^{\delta} + \delta \lVert R_n^{\delta} \rVert)}(x_n^{\delta})
   \end{displaymath}
   by
   \begin{displaymath}
    \tilde{x}_{n+1}^{\delta} = x_n^{\delta} - \frac{\langle u_n^{\delta},x_n^{\delta} \rangle - \left(\alpha_n^{\delta}+ \delta \lVert R_n^{\delta} \rVert\right)}{\lVert u_n^{\delta} \rVert^2} u_n^{\delta} 
    = x_n^{\delta} - \frac{\lVert R_n^{\delta} \rVert \left( \lVert R_n^{\delta} \rVert - \delta \right)}{\lVert u_n^{\delta} \rVert^2} u_n^{\delta}.
   \end{displaymath}
   Then, for all $z \in M_{Ax=y}$ we have
   \begin{displaymath}
    \left\lVert z - \tilde{x}_{n+1}^{\delta} \right\rVert^2 = \left\lVert z - x_{n}^{\delta} \right\rVert^2 - \left(\frac{\lVert R_n^{\delta} \rVert (\lVert R_n^{\delta} \rVert-\delta)}{\lVert u_n^{\delta} \rVert}\right)^2.
   \end{displaymath}
   If $\tilde{x}_{n+1}^{\delta} \in H_{n-1}^{\delta}$, we have $\tilde{x}_{n+1}^{\delta} = P_{H_n^{\delta}\cap H_{n-1}^{\delta}}(x_n^{\delta})$, so we define $x_{n+1}^{\delta}:=\tilde{x}_{n+1}^{\delta}$ and are done. Otherwise go to step
  \item[(ii)] Depending on $\tilde{x}_{n+1}^{\delta} \in H_{\gtrless}(u_{n-1}^{\delta},\alpha_{n-1}^{\delta} \pm \delta \lVert R_{n-1}^{\delta} \rVert)$, compute 
   \begin{displaymath}
    x_{n+1}^{\delta} := P_{H(u_n^{\delta}, \alpha_{n}^{\delta}+\delta \lVert R_n^{\delta} \rVert) \cap H(u_{n-1}^{\delta}, \alpha_{n-1}^{\delta}\pm\delta \lVert R_{n-1}^{\delta} \rVert)}(\tilde{x}_{n+1}^{\delta}),
   \end{displaymath}
   i.~e.
   \begin{displaymath}
    x_{n+1}^{\delta} = \tilde{x}_{n+1}^{\delta} - t_{n,n}^{\delta} u_n^{\delta} - t_{n,n-1}^{\delta} u_{n-1}^{\delta}
   \end{displaymath}
   such that $(t_{n,n}^{\delta},t_{n,n-1}^{\delta})$ minimizes
   \begin{displaymath}
    h_2(t_1,t_2) := \frac{1}{2} \left\lVert x_n^{\delta} - t_1 u_n^{\delta} - t_2^{\delta} u_{n-1}^{\delta} \right\rVert^2 + t_1(\alpha_n^{\delta} + \delta \lVert R_n^{\delta} \rVert) + t_2(\alpha_{n-1}^{\delta} \pm \delta \lVert R_{n-1}^{\delta} \rVert).
   \end{displaymath}
   Then we have $x_{n+1}^{\delta} = P_{H_{n}^{\delta} \cap H_{n-1}^{\delta}}(x_n^{\delta})$ and for all $z \in M_{Ax=y}$ we have
   \begin{displaymath}
    \lVert z - x_{n+1}^{\delta} \rVert^2 \leq \lVert z - x_n^{\delta} \rVert - S_n^{\delta}
   \end{displaymath}
   with
   \begin{displaymath}
    S_n^{\delta} := \left( \frac{\lVert R_n^{\delta} \rVert(\lVert R_n^{\delta} \rVert- \delta)}{\lVert u_n^{\delta} \rVert} \right)^2 + \left( \frac{\big\lvert \langle u_{n-1}^{\delta}, \tilde{x}_{n+1}^{\delta} \rangle - (\alpha_{n-1}^{\delta} \pm \delta \lVert R_{n-1}^{\delta} \rVert) \big\rvert}{\gamma_n \lVert u_{n-1}^{\delta} \rVert} \right)^2
   \end{displaymath}
   and
   \begin{displaymath}
    \gamma_n := \left( 1-\left( \frac{\big\lvert \langle u_{n}^{\delta}, u_{n-1}^{\delta} \rangle \big\rvert}{\lVert u_{n}^{\delta} \rVert \lVert u_{n-1}^{\delta} \rVert} \right)^2 \right)^{\frac{1}{2}} \in (0,1].
   \end{displaymath}
 \end{itemize} 
\end{algorithm}

The advantage of this algorithm is that using the current gradient as a search direction assures that the descent property \eqref{descent2} holds, yielding convergence, while the use of the gradient from the latest step speeds up the descent additionally: the larger the second summand in $S_n^{\delta}$, the greater the descent.\\[1ex]


\section{Sequential Subspace Optimization for Nonlinear Inverse Problems} \label{section_nonlinear}
We now want to develop a method for nonlinear inverse problems based on the one we introduced in the section above and consider the operator equation
\begin{equation} \label{nonlinearproblem}
 F(x) = y,
\end{equation}
where $F: \mathcal{D}(F) \subset X \rightarrow Y$ is a nonlinear operator between Hilbert spaces $X$ and $Y$ and $\mathcal{D}(F)$ is its domain. Our goal is to translate the idea of sequential projections onto stripes to the context of nonlinear operators. For that purpose, we have to make sure that the solution set
\begin{equation}
 M_{F(x)=y} := \left\lbrace x \in X \ : \ F(x)=y \right\rbrace
\end{equation}
is included in any stripe onto which we project in an effort to approach a solution of \eqref{nonlinearproblem}. \\
A simple replacement of $A^*$ by the Fr{\' e}chet derivative $F'(x_n)^*$ at the current iterate $x_n$, as might be the first idea given that the current gradient plays an important role in SESOP methods, does generally not ensure that $M_{F(x)=y}$ is included in a hyperplane or stripe of the form $H\big(F'(x_n)^*w_{n,i},\alpha_i,\xi_i\big)$ as defined previously. The reason is 
obvious: a solution $x$ of \eqref{nonlinearproblem} is in general not mapped onto $y$ by $F'(\tilde{x})$ for some $\tilde{x} \in \mathcal{D}(F)$. Furthermore, the fact that the linearization $F'(x)$ depends on the position $x \in X$ shows that the local character of nonlinear problems will need to be taken into account when dealing with problems like \eqref{nonlinearproblem}. This is strongly reflected in the following assumptions on the properties of $F$.
\vspace*{2ex} \\
Let $F:\mathcal{D}(F)\subset X\to Y$ be continuous and Fr\'echet differentiable in a ball $B_{\rho}(x_0) \subset \mathcal{D}(F)$ centered about the initial value $x_0 \in \mathcal{D}(F)$ with radius $\rho > 0$ and let the Fr\'echet derivative $F' (\cdot)$ be continuous. We claim the existence of a solution $x^+ \in X$ of \eqref{nonlinearproblem} satisfying
\begin{equation} \label{existence_solution_in_ball}
 x^+ \in B_{\frac{\rho}{2}}(x_0).
\end{equation}
Furthermore, we postulate that $F$ fulfills the \emph{tangential cone condition}
\begin{equation} \label{tcc}
 \left\lVert F(x) - F(\tilde{x}) - F'(x)(x-\tilde{x}) \right\rVert \leq c_{\mathrm{tc}} \left\lVert F(x) - F(\tilde{x}) \right\rVert
\end{equation}
for a nonnegative constant $\ctc < 1$ and that there is a positive constant $c_F>0$ such that
\begin{equation} \label{frechet_bounded}
 \lVert F'(x) \rVert < c_F
\end{equation}
for all $x \in B_{\rho}(x_0)$. Also, we assume the operator $F$ to be weakly sequentially closed, i.~e.
\begin{equation}
 \left( x_n \rightharpoonup x \wedge F(x_n) \rightarrow y \right) \Rightarrow \left( x \in \mathcal{D}(F) \wedge F(x) = y \right).
\end{equation}
As before, in the case of perturbed data we assume a noise level $\delta$ and postulate
\begin{displaymath}
 \left\lVert y^{\delta} - y \right\rVert \leq \delta.
\end{displaymath}
The residual is again defined by
\begin{displaymath}
 R_n := F(x_n) - y
\end{displaymath}
for exact data and by
\begin{displaymath}
 R_n^{\delta} := F(x_n^{\delta}) - y^{\delta}
\end{displaymath}
for noisy data. For later convenience, we want to define the \emph{current gradient} 
\begin{equation}
 g_n^{\delta} := F'(x_n^{\delta})^*(F(x_n^{\delta})-y^{\delta})
\end{equation}
in iteration $n \in \NN$ as the gradient of the functional
\begin{displaymath}
 \frac{1}{2} \lVert F(x)-y^{\delta} \rVert^2
\end{displaymath}
evaluated at the current iterate $x_n^{\delta}$.
\vspace*{2ex} \\
The projection onto hyperplanes in the linear case for exact data is convenient as the solution set itself is a hyperplane in $X$, spanned by elements of the null space of $A$. When dealing with nonlinear problems, this is no longer true. Our approach is to consider stripes - similar to the ones in the RESESOP scheme for linear operators - already for unperturbed data, using the tangential cone condition.\\


\subsection{The case of exact data}

We formulate SESOP for nonlinear operators in case of exact data.\\[1ex]

\begin{definition}
 For $u \in X$, $\alpha \in \RR$ and $\xi \in \RR$ with $\xi > 0$ we define the \emph{stripe}
 \begin{equation}
  H(u,\alpha,\xi) := \big\lbrace x \in X \ : \ \big\lvert \langle u,x \rangle - \alpha \big\rvert \leq \xi \big\rbrace.
 \end{equation}
\end{definition}

\begin{algorithm} \label{algo4}
 Choose $x_0 \in X$ as an initial value. In step $n\in\NN$, choose a finite index set $I_n$ and define
 \begin{equation}
 H_{n,i} := H(u_{n,i},\alpha_{n,i},\xi_{n,i})
\end{equation}
 with
 \begin{equation} \label{choiceparametersn1}
  \begin{split}
   u_{n,i} &:= F'(x_i)^*w_{n,i}, \\
   \alpha_{n,i} &:= \langle F'(x_i)^*w_{n,i}, x_i \rangle - \langle w_{n,i}, F(x_i) - y \rangle, \\
   \xi_{n,i} &:= \ctc \lVert w_{n,i} \rVert \lVert R_i \rVert.
  \end{split}
 \end{equation}
 Calculate
 \begin{equation}
  x_{n+1} = x_n - \sum_{i \in I_n} t_{n,i} F'(x_i)^*w_{n,i},
 \end{equation}
 where $t_n:=(t_{n,i})_{i \in I_n}$ are chosen such that 
 \begin{equation}
  x_{n+1} \in \bigcap_{i \in I_n} H_{n,i},
 \end{equation}
 i.~e.~the new iterate is given by the projection
 \begin{displaymath}
  x_{n+1} = P_{\bigcap_{i \in I_n} H_{n,i}}(x_n) = \mathrm{argmin}_{x \in \bigcap_{i \in I_n} H_{n,i}} \lVert x_n - x \rVert.
 \end{displaymath}
\end{algorithm}

\begin{definition}
 We call
\begin{displaymath}
 U_n:= \left\lbrace u_{n,i} \ : \ i \in I_n \right\rbrace \subset X
\end{displaymath}
 the \emph{search space} at iteration $n \in \NN$.\\[1ex]
\end{definition}

Taking a closer look at the definition of the stripes, we see that we have replaced $A^*$ with the adjoint of the linearization of $F$ in the iterate $x_i$. The width of the stripe depends on the constant $\ctc$ from the cone condition \eqref{tcc}. The other alterations can be interpreted as a localization of the hyperplanes subject to the local properties of $F$ in a neighborhood of the initial value. This becomes clear when we write the stripe $H_{n,i}$ in the form
\begin{displaymath}
 H_{n,i} := \left\lbrace x \in X \ : \ \big\lvert \langle F'(x_i)^*w_{n,i}, x_i -x \rangle - \langle w_{n,i}, F(x_i) -y \rangle \big\rvert \leq \ctc \lVert w_{n,i} \rVert \lVert R_i \rVert  \right\rbrace,
\end{displaymath}
i.~e.~we have to work with distances $x_i-x$ and $F(x_i)-y$ to the current iterate, respectively the point of linearization, and the value of $F$ in $x_i$.\\[1ex]

\begin{proposition} \label{prop_solution_set_contained1}
 For any $n \in \NN$, $i \in I_n$, the solution set $M_{F(x)=y}$ fulfills
 \begin{displaymath}
  M_{F(x)=y} \subset H_{n,i},
 \end{displaymath}
 where $u_{n,i}$, $\alpha_{n,i}$ and $\xi_{n,i}$ are chosen as in \eqref{choiceparametersn1}.\\[1ex]
\end{proposition}

\begin{proof}
 Let $z \in M_{F(x)=y}$. We then have
 \begin{displaymath}
  \langle u_{n,i},z \rangle - \alpha_{n,i} = \langle w_{n,i}, F'(x_i)^*(z-x_i) + F(x_i) - y \rangle.
 \end{displaymath}
 With $F(z)=y$ we obtain
 \begin{align*}
  \big\lvert \langle w_{n,i}, F'(x_i)(z-x_i) -F(x_i) + F(z) \rangle \big\rvert &\leq \lVert w_{n,i} \rVert \cdot \lVert F(x_i) - F(z) - F'(x_i)(x_i - z)\rVert \\
     &\leq \ctc \lVert w_{n,i} \rVert \cdot \lVert F(x_i)-y \rVert
 \end{align*}
 and using $R_i=F(x_i)-y$ we have $z \in H_{n,i}$.
\end{proof}




\subsection{The case of noisy data}
Of course we want to extend our method to the case of noisy data. To this end, we again have to modify the stripes onto which we project, now taking into account the noise level. The following definition of stripes $H_{n,i}^{\delta}$ ensures that the solution set is contained in each stripe.\\[1ex]

\begin{definition} \label{definition_parameters_algo5}
 For $n \in \NN$ and $i \in I_n$ we define the stripes
 \begin{displaymath}
  H_{n,i}^{\delta} := H(u_{n,i}^{\delta},\alpha_{n,i}^{\delta},\xi_{n,i}^{\delta})
 \end{displaymath}
 with
 \begin{equation}
  \begin{split}
   u_{n,i}^{\delta} &:= F'(x_i^{\delta})^*w_{n,i}^{\delta},  \\
   \alpha_{n,i}^{\delta} &:= \left\langle F'(x_i^{\delta})^*w_{n,i}^{\delta}, x_i^{\delta} \right\rangle - \left\langle w_{n,i}^{\delta}, F(x_i^{\delta})-y^{\delta} \right\rangle, \\
   \xi_{n,i}^{\delta} &:= \left( \delta + \ctc(\lVert R_i^{\delta} \rVert + \delta) \right) \lVert w_{k,i}^{\delta} \rVert.
  \end{split}
 \end{equation}
\end{definition}

We can easily see that $M_{F(x)=y}$ is contained in each stripe $H_{n,i}^{\delta}$. As in the proof of Proposition \ref{prop_solution_set_contained1} we obtain for any $z\in M_{F(x)=y}$
\begin{displaymath}
\begin{split}
 \big\lvert \left\langle u_{n,i}^{\delta}, z \right\rangle - \alpha_{n,i}^{\delta} \big\rvert &= \big\lvert \left\langle w_{n,i}^{\delta}, F'(x_i^{\delta})(z - x_i^{\delta}) + F(x_i^{\delta}) - F(z) + F(z) - y^{\delta} \right\rangle \big\rvert \\
   &\leq \left\lVert w_{n,i}^{\delta} \right\rVert \left( \left\lVert F(x_i^{\delta}) - F(z) - F'(x_i^{\delta})(x_i^{\delta}-z) \right\rVert + \left\lVert y^{\delta} - y \right\rVert \right) \\
   &\leq \left\lVert w_{n,i}^{\delta} \right\rVert \cdot \left(\ctc \lVert F(x_i^{\delta}) -y^{\delta} + y^{\delta}- F(z)  \rVert + \delta \right) \\
   &\leq \left\lVert w_{n,i}^{\delta} \right\rVert \cdot \left(\ctc \left(\lVert R_i^{\delta} \lVert + \delta \right) + \delta \right).
 \end{split}
\end{displaymath}

\begin{algorithm} \label{algo5}
 Let $x_0^{\delta}$ be an initial value and $\tau>1$ a constant. At iteration $n \in \NN$, choose a finite index set $I_n^{\delta}$ and search directions $F'(x_i^{\delta})^*w_{n,i}^{\delta}$ such that $n \in I_n$ and $F'(x_n^{\delta})^*w_{n,n}^{\delta} = F'(x_n^{\delta})^*(F(x_n^{\delta})- y^{\delta})$. If the residual $R_n^{\delta}=F(x_n^{\delta})-y^{\delta}$ satisfies the discrepancy principle
 \begin{equation}
  \lVert R_n^{\delta} \rVert \leq \tau\delta,
 \end{equation}
 stop the iteration. Otherwise compute
 \begin{equation}
  x_{n+1}^{\delta} = x_n^{\delta} - \sum_{i \in I_n^{\delta}} t_{n,i}^{\delta} F'(x_i^{\delta})^*w_{n,i}^{\delta},
 \end{equation}
 where $t_n^{\delta} = (t_{n,i}^{\delta})_{i\in I_n^{\delta}}$ is chosen such that
 \begin{equation}
  x_{n+1}^{\delta} \in \bigcap_{i \in I_n^{\delta}} H(u_{n,i}^{\delta},\alpha_{n,i}^{\delta},\xi_{n,i}^{\delta}),
 \end{equation}
 where $u_{n,i}^{\delta}$, $\alpha_{n,i}^{\delta}$ and $\xi_{n,i}^{\delta}$ are defined as in Definition \ref{definition_parameters_algo5}, and such that an estimate of the form
 \begin{equation}
  \lVert z-x_{n+1}^{\delta} \rVert^2 \leq \lVert z-x_n^{\delta} \rVert - C \lVert R_n^{\delta} \rVert^2
 \end{equation}
 holds for all $z \in M_{F(x)=y}$ for some constant $C>0$.\\[1ex]
\end{algorithm}

In the following subsection, we want to introduce a special case of the above algorithm, where the choice of search directions not only provides a fast regularized solution of a nonlinear problem as in \eqref{nonlinearproblem}, but also admits a good understanding of the structure of the method.\\

\subsection{An algorithm with two search directions} \label{subsection_nsesop2}
In analogy to Algorithm \ref{algo3} we want to develop a fast method for nonlinear operator equations, where we use only two search directions in each iteration. For linear problems in Banach spaces, this method provides a fast algorithm to calculate a regularized solution of \eqref{linearproblem}, where the search space consists in every step of the gradients $g_{n}^{\delta}$ and $g_{n-1}^{\delta}$. In the first step, only the gradient $g_{n}^{\delta}$ is used, so that the first iteration is similar to a Landweber step. We will adapt this method for nonlinear inverse problems \eqref{nonlinearproblem} in Hilbert spaces and analyze the convergence of our proposed algorithm. Also we show that, together with the discrepancy principle as a stopping rule, we obtain a regularization method for the solution of \eqref{nonlinearproblem}.\\[1ex]

\begin{remark}
 When projecting an element $x \in X$ onto the intersection of two halfspaces (or stripes) $H_1$ and $H_2$ by first calculating $x_1 = P_{H_1}(x)$ and afterwards $x_2=P_{H_1\cap H_2}(x_1)$, we may have $x_2 \neq P_{H_1 \cap H_2}(x)$. This can occur if $x$ is contained in neither $H_1$ nor $H_2$. If $x$ is already contained in $H_1$ or $H_2$, we have equality. This has been illustrated in \cite{ss09} for Hilbert spaces. The reason is that the order of projection plays an important role. If we have for example $x\in H_1$, the order of projections is evident and yields the desired result. We want to exploit this fact for our algorithm.\\[1ex]
\end{remark}

\begin{algorithm} (RESESOP for Nonlinear Operators with Two Search Directions) \label{algo6}
 Select an initial value $x_0^{\delta}:=x_0$. For $n=0$ choose $u_0^{\delta}$ and, if $n\geq 1$, choose $\left\lbrace u_n^{\delta}, u_{n-1}^{\delta} \right\rbrace$ as search directions in Algorithm \ref{algo5} with
 \begin{equation}
 \begin{split}
  u_n^{\delta} &:= F'(x_n^{\delta})^*w_n^{\delta}, \\
  w_n^{\delta} &:= F(x_n^{\delta})-y^{\delta}.
 \end{split}
 \end{equation}
 Define $H_{-1}^{\delta}:=X$ and for $n\in\NN$ the stripe
 \begin{equation}
  H_n^{\delta} := H(u_n^{\delta},\alpha_n^{\delta},\xi_n^{\delta})
 \end{equation}
 with
 \begin{equation}
 \begin{split}
  \alpha_n^{\delta} &:= \left\langle u_n^{\delta}, x_n^{\delta} \right\rangle - \lVert R_n^{\delta} \rVert^2, \\
  \xi_n^{\delta} &:= \lVert R_n^{\delta} \rVert \big( \delta + \ctc(\lVert R_n^{\delta} \rVert +\delta) \big).
 \end{split}
 \end{equation}
 As a stopping rule choose the discrepancy principle, where
 \begin{equation} \label{choice_tau_algo6}
  \tau > \frac{1+\ctc}{1-\ctc}.
 \end{equation}
 As long as $\lVert R_n^{\delta} \rVert > \tau\delta$, we have
 \begin{equation} \label{situation_for_xn}
  x_n^{\delta} \in H_>(u_n^{\delta},\alpha_n^{\delta} + \xi_n^{\delta}) \cap H_{n-1}^{\delta},
 \end{equation}
 and thus calculate the new iterate $x_{n+1}^{\delta}$ according to the following two steps:
 \begin{itemize}
  \item[(i)] Calculate
  \begin{equation} 
  \begin{split}
   \tilde{x}_{n+1}^{\delta} &:= P_{H(u_n^{\delta},\alpha_n^{\delta}+\xi_n^{\delta})} (x_n^{\delta}) \\
                            &= x_n^{\delta} - \frac{\left\langle u_n^{\delta}, x_{n}^{\delta} \right\rangle-\left( \alpha_n^{\delta}+\xi_n^{\delta} \right)}{\left\lVert u_n^{\delta} \right\rVert^2} u_n^{\delta}.
  \end{split}
  \end{equation}
  Then an estimate of the form
  \begin{displaymath}
   \left\lVert z-\tilde{x}_{n+1}^{\delta} \right\rVert^2 \leq \left\lVert z-x_n^{\delta} \right\rVert^2 - \left( \frac{\lVert R_n^{\delta} \rVert\left( \lVert R_n^{\delta} \rVert - \delta - \ctc (\lVert R_n^{\delta} \rVert +\delta) \right)}{\left\lVert u_n^{\delta} \right\rVert} \right)^2
  \end{displaymath}
  is valid for all $z \in M_{F(x)=y}$. \\
  If we have $\tilde{x}_{n+1}^{\delta} \in H_{n-1}^{\delta}$, we have $\tilde{x}_{n+1}^{\delta}=P_{H_n^{\delta}\cap H_{n-1}^{\delta}}(x_n^{\delta})$ and we are done. Otherwise, go to step (ii): 
  \item[(ii)] First, decide whether $\tilde{x}_{n+1}^{\delta} \in H_>(u_{n-1}^{\delta}, \alpha_{n-1}^{\delta} + \xi_{n-1}^{\delta})$ or $\tilde{x}_{n+1}^{\delta} \in H_<(u_{n-1}^{\delta}, \alpha_{n-1}^{\delta} - \xi_{n-1}^{\delta})$. Calculate accordingly
  \begin{displaymath}
   x_{n+1}^{\delta} := P_{H(u_n^{\delta},\alpha_n^{\delta}+\xi_n^{\delta}) \cap H(u_{n-1}^{\delta}, \alpha_{n-1}^{\delta} \pm \xi_{n-1}^{\delta})}(\tilde{x}_{n+1}^{\delta}),
  \end{displaymath}
  i.~e.~determine $x_{n+1}^{\delta} = \tilde{x}_{n+1}^{\delta} - t_{n,n}^{\delta} u_n^{\delta} - t_{n,n-1}^{\delta}u_{n-1}^{\delta}$ such that $\left( t_{n,n}^{\delta},t_{n,n-1}^{\delta} \right)$ minimizes the function
  \begin{displaymath}
   h_2(t_1,t_2) := \frac{1}{2} \left\lVert \tilde{x}_{n+1}^{\delta} - t_1 u_n^{\delta} - t_2 u_{n-1}^{\delta} \right\rVert^2 + t_1 \left( \alpha_n^{\delta}+\xi_n^{\delta} \right) + t_2 \left( \alpha_{n-1}^{\delta} \pm \xi_{n-1}^{\delta} \right).
  \end{displaymath}
  Then we have $x_{n+1}^{\delta}=P_{H_n^{\delta} \cap H_{n-1}^{\delta}}(x_n^{\delta})$ and for all $z \in M_{F(x)=y}$ the descent property
  \begin{equation}
   \left\lVert z-x_{n+1}^{\delta} \right\rVert^2 \leq \left\lVert z-x_n^{\delta} \right\rVert^2 - S_n^{\delta}
  \end{equation}
  is fulfilled, where
  \begin{displaymath}
  \begin{split}
   S_n^{\delta} &:= \left( \frac{\lVert R_n^{\delta} \rVert\left( \lVert R_n^{\delta} \rVert - \delta - \ctc (\lVert R_n^{\delta} \rVert +\delta) \right)}{\left\lVert u_n^{\delta} \right\rVert} \right)^2 \\
   &\quad+ \left( \frac{\big\lvert \left\langle u_{n-1}^{\delta}, \tilde{x}_{n+1}^{\delta} \right\rangle - \left(\alpha_{n-1}^{\delta} \pm \xi_{n-1}^{\delta}\right) \big\rvert}{\gamma_n \left\lVert u_{n-1}^{\delta} \right\rVert} \right)^2
  \end{split}
  \end{displaymath}
  and
  \begin{displaymath}
   \gamma_n := \left( 1-\left( \frac{\big\lvert \left\langle u_n^{\delta}, u_{n-1}^{\delta} \right\rangle \big\rvert}{\left\lVert u_n^{\delta} \right\rVert \cdot \left\lVert u_{n-1}^{\delta} \right\rVert} \right)^2 \right)^{\frac{1}{2}} \quad \in (0,1].\vspace{2mm}
  \end{displaymath}
 \end{itemize}
\end{algorithm}

The validity of the statements made in Algorithm \ref{algo6} are a consequence of Proposition \ref{prop_proj_onto_two}. By projecting first onto the stripe $H_n^{\delta}$, we make sure that a descent property holds to guarantee weak convergence.\\[1ex]

\begin{remark} \label{remark_two_dir}
\begin{itemize}
 \item[(a)] As we have mentioned before, calculating the projection of $x\in X$ onto the intersection of two halfspaces or stripes by first projecting onto one of them and then projecting onto the intersection does not necessarily lead to the correct result if $x$ is not contained in at least one of them. In our algorithm we avoid this problem: By choosing $\tau$ according to \eqref{choice_tau_algo6}, we guarantee that in iteration $n$ (provided the iteration has not been stopped yet) the iterate $x_n^{\delta}$ is an element of $H_>(u_n^{\delta},\alpha_n^{\delta} + \xi_n^{\delta}) \cap H_{n-1}^{\delta}$, which determines the order of projection that leads to the desired result. \\
 As $x_n^{\delta}$ is the projection of $x_{n-1}^{\delta}$ onto $H_{n-1}^{\delta} \cap H_{n-2}^{\delta}$, we have $x_n^{\delta} \in H_{n-1}^{\delta}$. To see that \eqref{situation_for_xn} is valid as long as $\lVert R_n^{\delta} \rVert >\tau\delta$, we note that from \eqref{choice_tau_algo6} we have 
 \begin{displaymath}
  \lVert R_n^{\delta} \rVert >\tau\delta > \delta \frac{1+\ctc}{1-\ctc}.
 \end{displaymath}
 With $0 \leq \ctc < 1$ we obtain
 \begin{displaymath}
   \lVert R_n^{\delta} \rVert - \ctc\lVert R_n^{\delta} \rVert -\delta\ctc -\delta  > 0,
 \end{displaymath}
 yielding
 \begin{displaymath}
  \alpha_n^{\delta} + \xi_n^{\delta} = \langle u_n^{\delta}, x_n^{\delta} \rangle - \lVert R_n^{\delta} \rVert \cdot \left( \lVert R_n^{\delta} \rVert - \ctc\lVert R_n^{\delta} \rVert -\delta\ctc -\delta \right) < \langle u_n^{\delta}, x_n^{\delta} \rangle.
 \end{displaymath}
 Thus $x_n^{\delta} \in H_>(u_n^{\delta}, \alpha_n^{\delta} + \xi_n^{\delta})$ and we obtain \eqref{situation_for_xn}.\\
 \item[(b)] The choice \eqref{choice_tau_algo6} for $\tau$ depends strongly on the constant $\ctc$ from the tangential cone condition. The smaller $\ctc$, the better the approximation of $F$ by its linearization. Of course $\ctc=0$ implies the linearity of $F$.\\
 \item[(c)] As already stated in \cite{skhk12}, the improvement due to Step (ii) might be significant, if the search directions $u_n^{\delta}$ and $u_{n-1}^{\delta}$ fulfill
 \begin{displaymath}
  \frac{\big\lvert \left\langle u_n^{\delta}, u_{n-1}^{\delta} \right\rangle \big\rvert}{\left\lVert u_n^{\delta} \right\rVert \cdot \left\lVert u_{n-1}^{\delta} \right\rVert} \approx 1,
 \end{displaymath}
 as in that case the coefficient $\gamma_n$ is quite small and therefore $S_n^{\delta}$ is large. This can be illustrated by looking at the situation where $u_n^{\delta}\bot u_{n-1}^{\delta}$: The projection of $x_n^{\delta}$ onto $H_n^{\delta}$ is already contained in $H_{n-1}^{\delta}$, such that Step (ii) will not lead to any further improvement.\\
 \item[(d)] Algorithm \ref{algo6} is very useful for implementation: First of all, the search direction $u_{n-1}^{\delta}$ has already been calculated for the precedent iteration and can be reused. Also, the residual $R_n^{\delta}$ was necessarily calculated to see if the discrepancy principle is fulfilled and can be reused as we have $w_n^{\delta}:=F(x_n^{\delta})-y^{\delta} = R_n^{\delta}$. So the only costly calculation is to determine $F'(x_n^{\delta})^*R_n^{\delta}$. In some applications, for example in parameter identification, this corresponds to an evaluation of a partial differential equation. The effort is thus comparable to Landweber type iteration, but the algorithm may be faster as discussed in the previous point.\\[1ex]
\end{itemize}
\end{remark}


\section{Convergence and regularization properties} \label{section_convergence_regularization}
Finally we want to analyze the methods presented in Section \ref{section_nonlinear}. Using the conditions we postulated at the beginning of Section \ref{section_nonlinear}, we will show convergence results for the SESOP respective RESESOP algorithms which we adapted to solving nonlinear inverse problems. For the method with two search directions, Algorithm \ref{algo6}, we will prove that it yields a regularized solution of the nonlinear problem with noisy data.\\[1ex]

\begin{proposition} \label{prop_descent_property1}
 If we choose the search direction $u_{n,n} = F'(x_n)^*w_{n,n}$ as the current gradient $g_n$, i.~e.~ $g_n=u_{n,n}$ with $w_{n,n}=F(x_n)-y$
 and exact data $y\in Y$, then we have
 \begin{displaymath}
  x_n \in H_>(u_{n,n}, \alpha_{n,n}+\xi_{n,n}),
 \end{displaymath}
 where $\alpha_{n,n}$ and $\xi_{n,n}$ are chosen as in Algorithm \ref{algo4}. By projecting $x_n$ first onto $H(u_{n,n}, \alpha_{n,n}, \xi_{n,n})$, we obtain
 \begin{equation} \label{descent_exact_general}
  \lVert z-x_{n+1} \rVert^2 \leq \lVert z-P_{H(u_{n,n}, \alpha_{n,n}, \xi_{n,n})}(x_n) \rVert^2 \leq \lVert z-x_n \rVert^2 - \frac{(1-\ctc)^2}{\lVert F'(x_n) \rVert^2} \lVert R_n \rVert^2
 \end{equation}
 for $z\in M_{F(x)=y}$.\\[1ex]
\end{proposition}

\begin{proof}
 According to our choice of $w_{n,n}=F(x_n)-y=R_n$ we have
 \begin{displaymath}
  \begin{split}
   \alpha_{n,n} &= \langle u_{n,n},x_n \rangle - \lVert R_n \rVert^2, \\
   \xi_{n,n} &= \ctc \lVert R_n \rVert^2.
  \end{split}
 \end{displaymath}
 We thus have $\langle u_{n,n}, x_n \rangle - \alpha_{n,n} = \lVert R_n \rVert^2 > \xi_{n,n}$ as $0 < \ctc < 1$. The descent property is easily obtained with the help of \eqref{orthogonal_projection} and the estimate $\lVert F'(x_n)^*(F(x_n)-y) \rVert \leq \lVert F'(x_n) \rVert \cdot \lVert R_n \rVert$.\hfill
\end{proof}
\vspace{2mm}

\begin{proposition} \label{prop_weak_conv1}
 The sequence of iterates $\left\lbrace x_n \right\rbrace_{n\in\NN}$, generated by Algorithm \ref{algo4}, fulfills
 \begin{displaymath}
  x_n \in B_{\rho}(x_0)
 \end{displaymath}
 and has a subsequence $\left\lbrace x_{n_k} \right\rbrace_{k \in \NN}$ that converges weakly to a solution of \eqref{nonlinearproblem}, if we choose $w_{n,n}=F(x_n)-y$ for every iteration $n \in \NN$.\\[1ex]
\end{proposition}

\begin{proof}
 According to Equation \eqref{existence_solution_in_ball}, there is a solution $x \in B_{\frac{\rho}{2}}(x_0) \subset \mathcal{D}(F)$, such that $\lVert x - x_0 \rVert \leq \frac{\rho}{2}$. We use Proposition \ref{prop_descent_property1} and obtain $x_1 \in B_{\frac{\rho}{2}}(x_0)$ as
 \begin{displaymath}
  \lVert x - x_1 \rVert^2 \leq \lVert x - x_0 \rVert^2 \leq \frac{\rho^2}{4}. 
 \end{displaymath}
 Thus, by induction, the descent property \eqref{descent_exact_general} yields $x_n \in B_{\frac{\rho}{2}}(x)$. We conclude
 \begin{displaymath}
  \lVert x_n - x_0 \rVert \leq \lVert x_n - x \rVert + \lVert x - x_0 \rVert \leq \rho,
 \end{displaymath}
 so that $x_n \in B_{\rho}(x_0)$ for all $n\in\NN$ and the sequence $\left\lbrace \lVert x_n - x \rVert \right\rbrace_{n\in \NN}$ is bounded and monotonically decreasing. We thus have a weakly convergent subsequence $\left\lbrace x_{n_k} \right\rbrace_{k\in\NN}$ with a limit $\hat{x}:=\lim_{k\rightarrow \infty} x_{n_k}$. It remains to show that $\hat{x}\in M_{F(x)=y}$. For that purpose, we use again the descent property \eqref{descent_exact_general} and use the estimate \eqref{frechet_bounded} which is valid in $B_{\rho}(x_0)$,
 \begin{displaymath}
  \lVert \hat{x} - x_{n_k} \rVert^2 - \lVert \hat{x} - x_{n_{k+1}} \rVert^2 \geq \frac{(1-\ctc)^2}{\lVert F'(x_{n_k})\rVert^2} \lVert R_{n_k} \rVert^2 \geq \frac{(1-\ctc)^2}{c_F^2} \lVert R_{n_k} \rVert^2.
 \end{displaymath}
 Let $K \in \NN$ be an arbitrary index. The subsequence $\left\lbrace x_{n_k} \right\rbrace_{k\in\NN}$ fulfills
 \begin{displaymath}
 \begin{split}
  \sum_{k=0}^{K} \lVert R_{n_k} \rVert^2 &\leq \left(\frac{c_F}{1-\ctc}\right)^2 \sum_{k=0}^K \big(\lVert \hat{x} - x_{n_k} \rVert^2 - \lVert \hat{x} - x_{n_{k+1}} \rVert^2 \big) \\
                                         &= \left(\frac{c_F}{1-\ctc}\right)^2 \cdot \big(\lVert \hat{x} - x_{n_0} \rVert^2 - \lVert \hat{x} - x_{n_{K+1}} \rVert^2 \big) \\
                                         &\leq \left(\frac{c_F}{1-\ctc}\right)^2 \cdot \lVert \hat{x} - x_{n_0} \rVert^2  \quad < \infty.
 \end{split}
 \end{displaymath}
 This remains true for the limit $K\rightarrow\infty$, yielding the (absolute) convergence of the series $\sum_{k=0}^{\infty} \lVert R_{n_k} \rVert^2$. The sequence $\left\lbrace\lVert R_{n_k} \rVert\right\rbrace_{k\in\NN}$ has to be a null sequence, i.~e.~$\lVert F(x_{n_k})-y \rVert \rightarrow 0$ for $k\rightarrow\infty$. As $F$ is continuous and weakly sequentially closed, we have $F(\hat{x})=\lim_{k\rightarrow\infty}F(x_{n_k})=y$.\hfill
\end{proof}
\vspace{2mm}

\begin{theorem} \label{theorem1}
  Let $N \in \NN$ be a fixed integer and $g_n \in U_n$ for each $n\in \NN$.
  If $x^+ \in B_{\frac{\rho}{2}}(x_0)$ is the unique solution of \eqref{nonlinearproblem} in $B_{\rho}(x_0)$, the sequence of iterates $\lbrace x_n \rbrace_{n\in\NN}$ generated by Algorithm \ref{algo4} converges strongly to $x^+$, if we choose 
  \begin{displaymath}
    I_n \subset \left\lbrace n-N+1,...,n \right\rbrace
  \end{displaymath}
   and set
   \begin{displaymath}
    w_{n,i} := R_i = F(x_i)-y
   \end{displaymath}
   for each $i \in I_n$ and if the optimization parameters $t_{n,i}$ fulfill
  \begin{equation}\label{tni-bound}
   \lvert t_{n,i} \rvert \leq t
  \end{equation}
  for some $t > 0$.\\[1ex]
\end{theorem}

\begin{proof}
 Inspired by the proof of Theorem 2.3 from \cite{hns_cl} we will show that the sequence $\lbrace x_n \rbrace_{n\in\NN}$ is a Cauchy sequence. For that purpose, we define the sequence $\left\lbrace a_n \right\rbrace_{n\in\NN}$ where
 \begin{displaymath}
  a_n := x_n - x^+
 \end{displaymath}
 and show that $\left\lbrace a_n \right\rbrace_{n\in\NN}$ is a Cauchy sequence. \\
 As $F'(x_n)^*w_{n,n} = g_n$ in Algorithm \ref{algo6}, we can apply Proposition \ref{prop_weak_conv1}. We have seen in the respective proof that $\left\lbrace \lVert a_n \rVert \right\rbrace_{n\in\NN}$ is a bounded monotonically decreasing sequence. So we have
 \begin{displaymath}
  \lVert a_n \rVert \rightarrow \epsilon \quad \text{for} \ n \rightarrow \infty,
 \end{displaymath}
 for some $\epsilon\geq 0$. Let $j\geq n$ and choose the index $l$ with $n \leq l \leq j$ such that
 \begin{displaymath}
 \lVert R_l \rVert = \lVert F(x_l) - y \rVert \leq \lVert F(x_i) - y \rVert = \lVert R_i \rVert \quad \text{for all} \ n \leq i \leq j.  
 \end{displaymath}
 We have $\lVert a_j - a_n \rVert \leq \lVert a_j - a_l \rVert + \lVert a_l - a_n \rVert$ and
 \begin{align*}
  \lVert a_j - a_l \rVert^2 &= 2\left\langle a_l-a_j, a_l \right\rangle + \lVert a_j \rVert^2 - \lVert a_l \rVert^2, \\
  \lVert a_l - a_n \rVert^2 &= 2\left\langle a_l-a_n, a_l \right\rangle + \lVert a_n \rVert^2 - \lVert a_l \rVert^2.
 \end{align*}
 When $n \rightarrow \infty$, we have $\lVert a_j \rVert^2 \rightarrow \epsilon$, $\lVert a_l \rVert^2 \rightarrow \epsilon$ and $\lVert a_n \rVert^2 \rightarrow \epsilon$. To prove $\lVert a_j - a_l \rVert \rightarrow 0$ for $n \rightarrow \infty$, we have to show that $2\left\langle a_l-a_j, a_l \right\rangle \rightarrow 0$ for $n\rightarrow \infty$. To this end, we note that
 \begin{displaymath}
 \begin{split}
  a_j - a_l &= x_j - x_l \\
            &= x_{j-1} - x_l - \sum_{i \in I_{j-1}} t_{j-1,i} F'(x_i)^*w_{j-1,i} \\
            &= - \sum_{k=l}^{j-1} \sum_{i \in I_k} t_{k,i} F'(x_i)^*w_{k,i}
 \end{split}
 \end{displaymath}
 and obtain
 \begin{displaymath}
  \begin{split}
   \left\lvert \left\langle a_l-a_j, a_l \right\rangle \right\rvert &=  \left\lvert \left\langle \sum_{k=l}^{j-1} \sum_{i \in I_k} t_{k,i} F'(x_i)^*w_{k,i} , a_l \right\rangle \right\rvert \\
    &\leq \sum_{k=l}^{j-1} \sum_{i \in I_k} \left\lvert t_{k,i} \right\rvert \left\lvert \left\langle F'(x_i)^*w_{k,i}, a_l \right\rangle \right\rvert \\
    &\leq \sum_{k=l}^{j-1} \sum_{i \in I_k} \left\lvert t_{k,i} \right\rvert \left\lVert w_{k,i} \right\rVert \left\lVert F'(x_i)(x_l-x^+) \right\rVert.
  \end{split}
 \end{displaymath}
 We estimate the last factor by
 \begin{displaymath}
  \begin{split}
   \left\lVert F'(x_i)(x_l-x^+) \right\rVert &\leq \left\lVert y - F(x_i) + F'(x_i)(x_i-x^+) \right\rVert \\
                                             &\quad  +  \left\lVert F(x_i) -F(x_l) - F'(x_i)(x_i -x_l) \right\rVert + \left\lVert F(x_l) - y \right\rVert \\
                                             &\leq \ctc \left\lVert F(x_i) - y \right\rVert + \ctc \left\lVert F(x_i) - F(x_l) \right\rVert + \left\lVert F(x_l) - y \right\rVert \\
                                             &\leq 2\ctc \left\lVert F(x_i) - y \right\rVert + (1+\ctc) \left\lVert F(x_l) - y \right\rVert \\
                                             &\leq (3\ctc+1) \left\lVert F(x_i) - y \right\rVert.
  \end{split}
 \end{displaymath}
 The choice $w_{k,i} = F(x_i) - y$ yields
 \begin{displaymath}
 \begin{split}
  \left\lvert \left\langle a_l-a_j, a_l \right\rangle \right\rvert &\leq (3\ctc + 1)\sum_{k=l}^{j-1} \sum_{i\in I_k} \left\lvert t_{k,i} \right\rvert \left\lVert F(x_i) -y \right\rVert^2 \\
   &\leq (3\ctc + 1)t\sum_{k=l}^{j-1} \sum_{i\in I_k}  \left\lVert F(x_i) -y \right\rVert^2.
 \end{split}
 \end{displaymath}
 Note that due to $I_n \subset \left\lbrace n-N+1,...,n \right\rbrace$ for a fixed $N\in \NN$, the sum over $i\in I_k$ is a finite sum and a similar calculation as in the proof of Proposition \ref{prop_weak_conv1} shows that the right-hand side of the above equation tends to 0 for $n \rightarrow \infty$. Consequently, the sequence $\left\lbrace \lVert a_n \rVert \right\rbrace_{n\in\NN}$ is a Cauchy sequence and the same holds for the sequence $\left\lbrace x_n \right\rbrace_{n\in\NN}$, which converges due to the weak sequential closedness of $F$ to the solution $x^+ \in B_{\rho}(x_0)$ as the sequence of residuals tends to 0 for $n\rightarrow \infty$.\hfill
\end{proof}
\vspace{2mm}

\begin{remark}
We obtain the strong convergence of the sequence from Proposition \ref{prop_weak_conv1}, if $X$ or $Y$ are finite-dimensional: As each subsequence of $\lbrace x_n \rbrace_{n\in\NN}$ is bounded, it contains a weakly convergent subsequence which, in a finite dimensional Hilbert space, is strongly convergent. Thus, the whole sequence converges strongly to a solution of \eqref{nonlinearproblem}, see also \cite{ss09} for more details.\\[1ex]
\end{remark}


We now want to deal with the sequences generated by Algorithms \ref{algo5} and \ref{algo6} for noisy data $y^{\delta}$.\\[1ex]

\begin{definition}
 For Algorithms \ref{algo5} and \ref{algo6} presented in Section \ref{section_nonlinear} we define the \emph{stopping index} 
 \begin{displaymath}
  n_*:=n_*(\delta) := \min \lbrace n\in \NN : \lVert R_n^{\delta} \rVert \leq \tau\delta \rbrace
 \end{displaymath}
 as the smallest iteration index, at which the discrepancy principle is fulfilled.\\[1ex]
\end{definition}

\begin{remark} \label{remark_tau_convergence}
 The following statements are valid, as long as the parameter $\tau$ used in the discrepancy principle is chosen such that 
 \begin{equation}
  \tau > \frac{1+\ctc}{1-\ctc}
 \end{equation}
 as postulated in Algorithm \ref{algo6} with two search directions. \\
 For simplicity, we define $x_n^{\delta}=x_{n_*}^{\delta}$ for all $n > n_*$.\\[1ex]
\end{remark}

\begin{lemma} \label{monotonicity_noisydata}
 If the gradient $g_n^{\delta}=F'(x_n^{\delta})^*(F(x_n^{\delta})-y^{\delta})$ of the functional $\frac{1}{2}\lVert F(x)-y^{\delta} \rVert^2$ in $x_n^{\delta}$ is included in die search space $U_n$ for $n\in \NN$ in Algorithm \ref{algo5}, the sequence $\left\lbrace \left\lVert z-x_n^{\delta} \right\rVert\right\rbrace_{n\in \NN}$ with $z \in B_{\rho}(x_0) \cap M_{F(x)=y}$ decreases monotonically.\\[1ex]
\end{lemma}

\begin{proof}
 As we have seen before, the choice of $\tau$ yields $x_n^{\delta} \in H_>(u_n^{\delta},\alpha_n^{\delta}+\xi^{\delta})$ for $n < n_*$ if we set $w_{n,n}^{\delta} := R_n^{\delta}$ in Definition \ref{definition_parameters_algo5} such that the current gradient is included in the search space. 
 By projecting $x_n^{\delta}$ first onto the stripe $H(u_n^{\delta}, \alpha_n^{\delta}, \xi_n^{\delta})$, we obtain
 \begin{displaymath}
  \tilde{x}_{n+1}^{\delta}:=P_{H_>(u_n^{\delta},\alpha_n^{\delta}+\xi^{\delta})}(x_n^{\delta}) = x_n^{\delta} - \frac{\left\langle u_n^{\delta}, x_n^{\delta} \right\rangle - \left( \alpha_n^{\delta} + \xi_n^{\delta} \right) }{\lVert u_n^{\delta} \rVert^2} u_n^{\delta},
 \end{displaymath}
 and with $\lVert z-x_{n+1}^{\delta} \rVert^2 \leq \lVert z-\tilde{x}_{n+1}^{\delta} \rVert^2$ the estimate
 \begin{equation} \label{estimate_noisydata}
  \lVert z-x_{n+1}^{\delta} \rVert^2 \leq \lVert z-x_n^{\delta}\rVert^2 - \left( \frac{\lVert R_n^{\delta} \rVert\left( \lVert R_n^{\delta} \rVert - \delta - \ctc (\lVert R_n^{\delta} \rVert +\delta) \right)}{\left\lVert u_n^{\delta} \right\rVert} \right)^2,
 \end{equation}
 which we have shown before in Subsection \ref{subsection_nsesop2}. This proves the monotonicity of the sequence $\left\{ \left\lVert z-x_n^{\delta} \right\rVert \right\}_{n\in \NN}$, which is constant for $n\geq n_*$.\hfill
\end{proof}
\vspace{2mm}

\begin{theorem} \label{theorem2}
 Provided the current gradient $g_n^{\delta}$ is contained in the search space $U_n$ and the parameter $\tau$ is chosen such that
 \begin{displaymath}
  \tau > \frac{1+\ctc}{1-\ctc},
 \end{displaymath}
 the discrepancy principle yields a finite stopping index $n_*=n_*(\delta)$.\\[1ex]
\end{theorem}

\begin{proof}
 Let us assume that the discrepancy principle is not satisfied for any iteration index $n \in \NN$. We then have $\delta  < \frac{1}{\tau}\lVert R_n^{\delta} \rVert$ for all $n\in \NN$. In that case, Lemma \ref{monotonicity_noisydata} holds and equation \eqref{estimate_noisydata} yields
 \begin{displaymath}
 \begin{split}
  \lVert z-x_{n+1}^{\delta} \rVert^2 &\leq \lVert z-x_n^{\delta}\rVert^2 - \left( \frac{\lVert R_n^{\delta} \rVert\left( \lVert R_n^{\delta} \rVert - \delta - \ctc (\lVert R_n^{\delta} \rVert +\delta) \right)}{\left\lVert F'(x_n^{\delta}) \right\rVert\cdot \left\lVert R_n^{\delta} \right\rVert} \right)^2 \\
  &\leq \lVert z-x_n^{\delta}\rVert^2 - \left( \frac{ \lVert R_n^{\delta} \rVert - \delta - \ctc (\lVert R_n^{\delta} \rVert +\delta)}{\left\lVert F'(x_n^{\delta}) \right\rVert} \right)^2 \\
  &\leq \lVert z-x_n^{\delta}\rVert^2 - \left(\frac{1-\ctc-\frac{1}{\tau}(1+\ctc)}{c_F^2}\right)^2 \cdot \lVert R_n^{\delta} \rVert^2,
 \end{split}
 \end{displaymath}
 where we have
 \begin{displaymath}
  1-\ctc-\frac{1}{\tau}(1+\ctc) > 0
 \end{displaymath}
 according to our choice of $\tau$. By a calculation as in the proof of Proposition \ref{prop_weak_conv1} we find that $\lbrace\lVert R_n^{\delta} \rVert\rbrace_{n\in\NN}$ is a null sequence. This is a contradiction to our assumption $\lVert R_n^{\delta} \rVert>\tau\delta$ for all $n\in \NN$. Consequently, there must be a finite stopping index $n_*$ fulfilling the discrepancy principle.\hfill
\end{proof}
\vspace{2mm}

We now want to deal with the method proposed in Subsection \ref{subsection_nsesop2}, where we again make use of the current gradient, such that the previous statements apply to the sequence of iterates generated by Algorithm \ref{algo6}. In particular, we know that Algorithm \ref{algo6} terminates after a finite number of iterations. It remains to show that the final iterate $x_{n_*(\delta)}^{\delta}$ is a regularized solution of the nonlinear operator equation \eqref{nonlinearproblem}, when only perturbed data are given.\\[1ex]

\begin{theorem} \label{regularization}
 Algorithm \ref{algo6} together with the discrepancy principle yields a regularized solution $x_{n_*}^{\delta}$ of the nonlinear problem \eqref{nonlinearproblem}, when only noisy data $y^{\delta}$ are given, i.~e.~we have
 \begin{equation}
  x_{n_*(\delta)}^{\delta} \rightarrow x^+ \quad \mathrm{for} \ \delta \rightarrow 0,
 \end{equation}
 if there is only one solution $x^+\in B_{\rho}(x_0)$ and if $\lim_{\delta\rightarrow 0} \left\lvert t_{n,i}^{\delta} \right\rvert < t$ for all $n\in\NN$ and $i\in I_n$ for some $t>0$.\\[1ex]
\end{theorem}

\begin{proof}
  We have to show that 
 \begin{equation}
  \left\lVert F(x_{n_*(\delta)}^{\delta}) - y \right\rVert \rightarrow 0
 \end{equation}
 if the noise level $\delta$ tends to 0. For that purpose, define the null sequence $\left\lbrace \delta_j \right\rbrace_{j\in\NN} \subset \RR_{>0}$, such that 
 \begin{equation}
  \left\lVert y^{\delta_j} - y \right\rVert \leq \delta_j
 \end{equation}
 holds for all $j\in\NN$. For each noise level $\delta_j$, we define the corresponding stopping index $n_*(\delta_j)$, such that
 \begin{equation}
  \left\lVert R_{n_*(\delta_j)}^{\delta_j} \right\rVert \leq \tau\delta_j.
 \end{equation}
 Note that due to our previous results we already know $x_n^{\delta_j} \in B_{\rho}(x_0)$ for a fixed index $n \in \NN$ and all $j\in \NN$, such that the sequence $\left\lbrace x_{n_*(\delta_j)} \right\rbrace_{j\in\NN}$ is bounded and has a weakly convergent subsequence. 
 Consider the sequence $\left\lbrace n_*(\delta_j) \right\rbrace_{j\in\NN}$. Let us assume that $n\in\NN$ is a finite accumulation point of this sequence. As in the proof of Theorem 2.4 in \cite{hns_cl}, we assume without loss of generality that $n_*(\delta_j) = n$ for all $j\in\NN$. We thus have
 \begin{equation} \label{discrepancy_principle_proof}
  \left\lVert F(x_n^{\delta_j}) - y^{\delta_j} \right\rVert \leq \tau \delta_j
 \end{equation}
 for all $j \in \NN$. Since we have fixed the index $n$, the iterate $x_n^{\delta_j}$ depends continuously on $y^{\delta_j}$ yielding
 \begin{displaymath}
  \begin{split}
   x_n^{\delta_j} \rightarrow x_n, \quad F(x_n^{\delta_j}) \rightarrow F(x_n)
  \end{split}
 \end{displaymath}
 for $j \rightarrow \infty$. The discrepancy principle is fulfilled at iteration $n$, so we obtain $F(x_n)=y$ from Equation \eqref{discrepancy_principle_proof}. As $x_n \in B_{\rho}(x_0)$ and $x^+$ is the only solution of \ref{nonlinearproblem} in $B_{\rho}(x_0)$, we have 
 \begin{displaymath}
   \left\lVert F\left(x_{n_*(\delta_j)}^{\delta_j}\right) - y \right\rVert \leq \left\lVert F\left(x_{n_*(\delta_j)}^{\delta_j}\right) - y^{\delta_j} \right\rVert + \left\lVert y^{\delta_j} - y \right\rVert   \leq (\tau +1) \delta_j \ \rightarrow 0
 \end{displaymath}
 and
 \begin{displaymath}
  x_{n_*(\delta_j)}^{\delta_j} \rightarrow x^+
 \end{displaymath}
 for $j \rightarrow \infty$ due to the weak sequential closedness of $F$.\\
 Let us now assume $n_*(\delta_j) \rightarrow \infty$ for $j\rightarrow\infty$. From Theorem \ref{theorem1} we recall that the sequence $\left\lbrace \lVert x_n - x^+ \rVert \right\rbrace_{n\in\NN}$ is a bounded and monotonically decreasing Cauchy sequence converging to 0. Subsequently for every $\epsilon > 0$ there is an $N_0 \in \NN$ such that $\lVert x_{N_0} - x^+ \rVert < \frac{\epsilon}{2}$. Furthermore, there is a $j_0 \in \NN$ such that, due to the continuity of the norm, we have
 \begin{equation}
  \left\lvert \lVert x_{N_0}^{\delta_j} - x^+ \rVert^2 - \lVert x_{N_0} - x^+ \rVert^2 \right\rvert < \frac{\epsilon}{2}
 \end{equation}
 for all $j > j_0$ on the one hand and also, according to our assumption,
 \begin{displaymath}
  n_*(\delta_j) \geq N_0 \quad \text{for all} \ j\geq j_0.
 \end{displaymath}
We thus have
 \begin{displaymath}
  \lVert x_{n_*(\delta_j)}^{\delta_j} - x^+\rVert^2 \leq \lVert x_{N_0}^{\delta_j} - x^+\rVert^2 
  \leq \lVert x_{N_0} - x^+\rVert^2 + \left\lvert \lVert x_{N_0}^{\delta_j} - x^+ \rVert^2 - \lVert x_{N_0} - x^+ \rVert^2 \right\rvert < \epsilon
 \end{displaymath}
 for all $j > j_0$, what completes the proof.\hfill
\end{proof}
\vspace{2mm}

\begin{corollary}
 Theorem \ref{regularization} remains valid for Algorithm \ref{algo5}, as long as we choose in every iteration $g_n^{\delta} \in U_n^{\delta}$ and $u_{n,i}^{\delta} = F'(x_i^{\delta})^*R_i^{\delta} \in U_n^{\delta}$ for each $i \in I_n$, where $I_n \subset \lbrace n-N+1,...,n \rbrace$ for a fixed $N\in\NN$, i.~e.~we choose the search directions among the recent $N$ gradients.\\[1ex]
\end{corollary}

\begin{remark}
 The choice $U_n:=\left\lbrace g_n \right\rbrace$ for all $n\in\NN$ yields a Landweber type method, where the current iterate is projected onto a stripe corresponding to the conventional Landweber direction, for which we can apply our results from above.\\[1ex]
\end{remark}

\section{Conclusions}
In this article, we presented an iterative method designed for the solution of nonlinear inverse problems. The main idea was to define (convex) sets, which contain the set of solutions, and successively project onto these subsets. For this purpose, we defined stripes for both the exact and the noisy case, considering the nonlinearity of the operator by making use of a tangential cone condition. The method allows many possibilities to choose the search spaces, while including the current gradient guarantees at least the existence of a weakly convergent subsequence. As a special choice of search spaces, we presented a regularizing method that uses two search directions in every iteration, while the structure of the occuring stripes is well understood, which gives rise to a fast algorithm as the order of projection is clearly defined and yields the steepest possible descent. \\
The next step will be the numerical evaluation of our proposed methods. Also, there are several possible theoretical extensions of this work for future research. Obviously our methods can easily be adapted to Banach spaces as it has been done for linear forward operators. Using the common notation in Banach spaces, the search spaces $U_n$ from Algorithm \ref{algo6} are spanned by $J_{q^*}^{X^*}\left(u_k^{\delta}\right)$ where $u_k^{\delta} := F'(x_k^{\delta})^*w_k^{\delta}$ and $w_k^{\delta}=J_Y(F(x_k^{\delta})-y^{\delta})$ for $k\in I_n$ and duality mappings $J_{q^*}^{X^*}:X^*\to X$, $J_Y:Y\to Y^*$. Another option might be a closer investigation of suitable search spaces, keeping in mind Remark \ref{remark_two_dir} (c), which motivates the definition of some sort of measure for the possible gain yielded by the use of a certain search direction.

\bibliographystyle{siam}
\bibliography{bib_nsesop}

\pagestyle{myheadings}
\thispagestyle{plain}
\markboth{A. Wald and T. Schuster}{Sequential Subspace Optimization for Nonlinear Inverse Problems}

\end{document}